\documentclass[10pt]{amsart}

\usepackage[utf8]{inputenc}
\usepackage[T1]{fontenc}
\usepackage[english]{babel}
\usepackage{amsmath,amssymb,amsfonts,amsthm}
\usepackage{amsmath,amsthm,amsfonts,amssymb, hyperref}
\usepackage[pdftex]{graphicx}
\usepackage{enumitem}

\newtheorem{theorem}{Theorem}[section]

\newtheorem{proposition}[theorem]{Proposition}
\newtheorem{corollary}[theorem]{Corollary}
\newtheorem{remark}[theorem]{Remark}
\newtheorem{lemma}[theorem]{Lemma}
\newtheorem{conjecture}{Conjecture}

\theoremstyle{definition}

\renewcommand{\ll}{\langle\langle}
\newcommand{\rr}{\rangle\rangle}

\newcommand{\bt}{\mathbf{t}}
\newcommand{\bn}{\mathbf{n}}
\newcommand{\mH}{\mathcal H}
\newcommand{\mE}{\mathcal E}
\newcommand{\mC}{\mathcal C}
\newcommand{\mL}{\mathcal L}
\newcommand{\mD}{\mathcal D}

\DeclareMathOperator{\im}{\mathrm{im}}

\begin{document}

\title[Steklov problem on differential forms]{Steklov problem on differential forms}

\begin{abstract}

In this paper we study spectral properties of Dirichlet-to-Neumann map on differential forms obtained by a slight modification of the definition due to Belishev and Sharafutdinov~\cite{BS}. The resulting operator $\Lambda$ is shown to be self-adjoint on the subspace of coclosed forms and to have purely discrete spectrum there.
We investigate properies of eigenvalues of $\Lambda$ and prove a Hersch-Payne-Schiffer type inequality relating products of those eigenvalues to eigenvalues of Hodge Laplacian on the boundary. Moreover, non-trivial eigenvalues of $\Lambda$ are always at least as large as eigenvalues of Dirichlet-to-Neumann map defined by Raulot and Savo~\cite{RS}. Finally, we remark that a particular case of $p$-forms on the boundary of $2p+2$-dimensional manifold shares a lot of important properties with the classical Steklov eigenvalue problem on surfaces. 
\end{abstract}

\author[Mikhail Karpukhin]{Mikhail Karpukhin}
\address{Department of Mathematics and Statistics, McGill University, Burnside Hall,
805 Sherbrooke Street West,
Montreal, Quebec,
Canada,
H3A 0B9}
\email{mikhail.karpukhin@mail.mcgill.ca}
\maketitle

\section{Introduction}

Let $M$ be a compact Riemannian manifold of dimension $n$ with smooth boundary $\partial M$. Recently there has been a lot of research dedicated to Steklov eigenvalue problem which is defined in the following way. Number $\sigma$ is called a {\em Steklov eigenvalue} of $M$ provided there exists a non-zero solution $u\in C^\infty(M)$ to the following problem
\begin{equation}
\left\{
   \begin{array}{rcl}
	\Delta u &=& 0\quad \mathrm{on}\,\, M,\\
	\partial_nu &=& \sigma u\quad \mathrm{on}\,\, \partial M.
   \end{array}
\right.
\end{equation}

Steklov eigenvalues coincide with eigenvalues of the Dirichlet-to-Neumann operator $\mD\colon C^\infty(\partial M)\to C^\infty(\partial M)$. Operator $\mD$ sends a function $v$ to normal derivative of its harmonic extension. Then $\mD$ is a self-adjoint elliptic pseudodifferential operator of order $1$, i.e. Steklov eigenvalues form a sequence tending to $+\infty$. For details we refer the reader to survey paper~\cite{GPsurvey} and references therein.

In the present paper we study Steklov eigenvalues on the space of differential forms on $M$. Several definitions of Dirichlet-to-Neumann operator are present in the literature, see e.g.~\cite{BS,JL,RS}.
Definition commonly used in spectral theory literature is due to Raulot and Savo~\cite{RS} and has an advantage of being a positive elliptic self-adjoint pseudodifferential operator of order 1. However, in literature on inverse problems different definitions of Dirichlet-to-Neumann map are used, see e.g. full Direchlet-to-Neumann map in~\cite{JL,SS} and definition due to Belishev and Sharafutdinov~\cite{BS} which was motivated by Maxwell equations. In the present paper we modify the latter to obtain a self-adjoint operator with purely discrete spectrum and study its eigenvalues. We plan to tackle spectral theory of the full Dirichlet-to-Neumann map in a subsequent article.

%

\section{Main results}
\label{Statements}

\subsection{Notations}

In the following $(M,g)$ is always assumed to be a smooth compact orientable manifold of dimension $n$ with smooth nonempty boundary $\partial M$. It seems that orientability is a purely technical condition that could be eliminated with further investigations. Theorem~\ref{HPS} below, however, requires orientability in an essential way.

Let $(X,h)$ be a compact Riemannian manifold, possibly with boundary. The space of smooth differential $p$-forms on $X$ will be denoted by $\Omega^p(X)$. By $\mE^p(X)\subset\mC^p(X)\subset\Omega^p(X)$ we denote the spaces of smooth exact and closed $p$-forms respectively. A letter $c$ in front of either of them denotes the prefix "co-", concatenation of the letters stands for intersection, e.g. $\mathcal C c\mathcal C^p(X)$ is the space of closed and coclosed $p$-forms which in the following will be denoted by $\mH^p(X)$. If $\partial X = \varnothing$ then $\mH^p(X)$ coincides with the space of harmonic forms, i.~e. the kernel of the Hodge-Laplace operator. 

However, if $\partial X\ne\varnothing$, those spaces are different and we refer to elements of $\mathcal H^p(X)$ as {\em harmonic fields} and reserve the term {\em harmonic form} for elements of $\ker\Delta$. 
Let $i\colon \partial X\to X$ be an embedding of the boundary and let $i_n$ denote contraction of a differential form with the outer unit normal vector field. The form $\omega\in\Omega^p(X)$ satisfies Dirichlet (resp. Neumann) boundary condition if $i^*\omega = 0$ (resp. $i_n\omega = 0$). We use subscripts $D$ and $N$ to indicate spaces of forms satisfying Dirichlet or Neumann boundary conditions. Finally, for $\omega\in\Omega^p(M)$ we denote by $\bt\omega,\bn\omega\in\Gamma(i^*\Omega^p(X))$ the tangent and normal parts of $\omega$ on the boundary, i.e. $\bt\omega$ is $i^*\omega$ considered as a section of $i^*\Omega^p(X)$ and $\bn\omega = dn\wedge i_n\omega$, where $dn$ is a 1-form, dual to the outer unit normal vector field. In practice, the only difference between $\bn\omega$ and $i_n\omega$ for example, is the way Hodge $*$-operator acts on them, see Proposition~\ref{inbn} below.

For a subspace $V\subset\Omega^p(X)$ we denote by $H^sV\subset H^s\Omega^p(X)$ the completion of $V$ with respect to the Sobolev $H^s$-norm. We write $L^2$ instead of $H^0$. For details on Sobolev norms, see e.g.~\cite{Schwarz} Section 1.3. We use angle brackets $\langle\cdot,\cdot\rangle$ to denote pointwise $L^2$-inner product, double angle brackets $\ll\cdot,\cdot\rr$ to denote integrated $L^2$-inner product, round brackets $(\cdot,\cdot)$ to denote the $H^{-s}\times H^s \to \mathbb{R}$ duality pairing and $||\cdot||_{H^s}$ to denote $H^s$-norm. Usually it is clear from the context whether we are working on the boundary or on the manifold itself. In cases where it needs clarification, we add subscript indicating the ambient space, e.g. $||\cdot||_{L^2(X)}$ or $||\cdot||_{H^{1/2}(\partial X)}$.

Finally, let us remind that for manifolds with boundary Green's formula states that for $\alpha,\beta\in H^1\Omega^p(M)$
\begin{equation}
\label{Green}
\int\limits_M\langle d\alpha,\beta\rangle\, dV = \int\limits_M\langle\alpha,\delta\beta\rangle\,dV + \int\limits_{\partial M}\langle i^*\alpha,i_n\beta\rangle\,dA = \int\limits_M\langle\alpha,\delta\beta\rangle\,dV + \int\limits_{\partial M} i^*\alpha\wedge *\bn\beta.
\end{equation}

\subsection{Maxwell equations} 
In the modern form, Maxwell equations are usually written on the language of differential forms on an orientable 3-dimensional Riemannian manifold, see~\cite{T}. In the exposition below we follow~\cite{KKL}. Maxwell equations have the following form
\begin{equation*}
\begin{split}
d\mE = -\partial_t\mathcal B,&\qquad d\mH = \partial_t\mathcal D, \\
\mathcal D(x,t) = *_\epsilon\mE(x,t),&\qquad \mathcal B(x,t) = *_\mu\mH(x,t), \\
d\mathcal B = 0,&\qquad d\mathcal D = 0, 
\end{split}
\end{equation*}
where $\mE$ and $\mH$ are 1-forms corresponding to electric and magnetic fields, $\mathcal B$ and $\mathcal D$ are 2-forms corresponding to magnetic flux density and electric displacements, $*_\epsilon$ and $*_\mu$ are Hodge operators for some metrics corresponding to electric permittivity and magnetic permeability. In case the 3-manifold has a boundary, there is a natural response operator $R$ that sends the component of electric field tangent to the boundary to the component of magnetic field tangent to the boundary. In paper~\cite{KKL} the authors study inverse problem of recovering the manifold $M$ given the response operator. 

Consider the simplest case $*_\epsilon = *_\mu = *$ and the time-harmonic solution to Maxwell equations,~i.e. the $t$ variable is separated and solutions depend on $t$ only via factor $e^{ik t}$ for a fixed angular frequency $k\in\mathbb{R}$. Then Maxwell equation for $\mE$ and $\mathcal B$ becomes
\begin{equation*}
\left\{
   \begin{array}{rcl}
	-ik\mathcal B &=& d\mE,\\
	d*\mathcal B &=& ik*\mE,\\
	d\mathcal B &=& 0.\\
   \end{array}
\right.
\end{equation*}
In terms of $\mE$ this system has form
\begin{equation}
\label{Maxwell}
\left\{
   \begin{array}{rcl}
	\Delta\mE &=& k^2\mE,\\
	\delta\mE &=& 0.\\
   \end{array}
\right.
\end{equation}
and the response operator sends $\bt \mE \mapsto \bt*\mathcal B = \frac{i}{k}*\bn d\mE$, i.e. it connects $\bt\mE$ with $\bn d\mE$. In the next section we use this calculation to motivate the definition of Dirichlet-to-Neumann map on differential forms for Riemannian manifolds of arbitrary dimension.

\subsection{Definition and basic properties}

Let $M$ be a compact orientable manifold with smooth non-empty boundary $\partial M$. Motivated by the particular case $k = 0$ of~(\ref{Maxwell}) we define Dirichlet-to-Neumann operator $\Lambda$ acting on the space of differential forms $\Omega^p(\partial M)$ in the following way. For any $\phi\in\Omega^p(\partial M)$ consider the equation
\begin{equation}
\label{DtN}
\left\{
   \begin{array}{rcl}
	\Delta\omega &=& 0,\\
	\delta\omega &=& 0,\\
	 i^* \omega &=& \phi.\\
   \end{array}
\right.
\end{equation}
Let us denote the space of solutions $\omega$ by $\mathcal L(\phi)$. In Proposition~\ref{L} below it is proved that $\mathcal L(\phi)$ is an affine vector space with an associated vector space $\mH^p_D(M)$. We set $\Lambda\phi := i_nd\omega$ for any $\omega\in\mathcal L(\phi)$. Since $d\mathcal H^p_D(M) = 0$, definition does not depend on the choice of $\omega$. Let us denote by $\lambda(\phi)\in\mathcal L (\phi)$ the unique solution of~(\ref{DtN}) satisfying $\lambda(\phi)\perp\mH^p_D(M)$.

\begin{remark}
In~\cite{BS} the Dirichlet-to-Neumann map is defined up to a sign as $*\Lambda$.
\end{remark}
\begin{remark}
Having in mind equation~(\ref{Maxwell}), it is more natural to consider the operator $\Lambda(\lambda)$ for $\lambda\in\mathbb{R}$ defined in the same way as $\Lambda$ but instead of~(\ref{DtN}) one requires $\omega$ to be the solution of
\begin{equation*}
\left\{
   \begin{array}{rcl}
	\Delta\omega &=& \lambda\omega,\\
	\delta\omega &=& 0,\\
	 i^* \omega &=& \phi.\\
   \end{array}
\right.
\end{equation*}
However, the study of $\Lambda(\lambda)$ for $\lambda\ne 0$ exceeds the scope of the present article.
\end{remark}
Our starting point is the following theorem.
\begin{theorem}
\label{MainTheorem}
Operator $\Lambda$ is identically zero on the space $\mE^p(\partial M)$. Restricted to the space $c\mathcal C^p(\partial M)$ it is a positive self-adjoint operator with compact resolvent. In particular, its spectrum is discrete and is denoted by 
$$
0\leqslant \sigma^{(p)}_1\leqslant \sigma^{(p)}_2\leqslant\ldots \nearrow\infty,
$$
where the eigenvalues are written with multiplicity and all multiplicities are finite. The kernel satisfies $\ker\Lambda\cap c\mathcal C^p(\partial M) = i^*\mathcal{H}^p_N(M)\cap c\mathcal C^p(\partial M)$ and has dimension $I_p = \dim\mathrm{im}\{i^*\colon H^p(M)\to H^p(\partial M)\}$.

Moreover, the eigenvalues can be characterised by the following min-max formula,
$$
\sigma_k^{(p)} = \max_E \min_{\phi\perp E;\, i^*\hat\phi = \phi} \frac{||d\hat\phi||^2_{L^2(M)}}{||\phi||^2_{L^{2}(\partial M)}},
$$
where $E$ runs over all $(k-1)$-dimensional subspaces of $c\mC^p(\partial M)$. Maximum is achieved for $E = V_{k-1}$, where $V_{k-1}$ is spanned by the first $(k-1)$ eigenforms, $\phi$ being the $k$-th eigenform and $\hat\phi\in\mathcal L(\phi)$.
\end{theorem}

\begin{remark} An alternative way to prove the first part of Theorem~\ref{MainTheorem} is to show that 
 $\Lambda |_{c\mathcal C^p}$ is an elliptic pseudodifferential operator. We intend to explore this route in a subsequent paper.
\end{remark}

\subsection{Main results} 
Our main results are concerned with properties of eigenvalues of $\sigma^{(p)}_k$.
First, we prove a comparison theorem between eigenvalues of $\Lambda$ and eigenvalues of the Dirichlet-to-Neumann map $L$ defined by Raulot and Savo in~\cite{RS}. For any $\phi\in\Omega^p(\partial M)$ there exists a unique solution $\omega$ to the following problem (see Theorem 3.4.10 in~\cite{Schwarz}),

\begin{equation}
\left\{
   \begin{array}{rcl}
	\Delta\omega &=& 0,\\
	i_n\omega &=& 0,\\
	 i^* \omega &=& \phi.\\
   \end{array}
\right.
\end{equation}
Then $L(\phi)$ is defined to be equal to $i_nd\omega$. Moreover, $L$ is an elliptic pseudodifferential operator of order 1, so its spectrum is discrete and is denoted by
$$
0\leqslant \mu^{(p)}_1\leqslant \mu^{(p)}_2\leqslant\ldots \nearrow\infty.
$$
We also use notations $\tilde\mu_i^{(p)}$ and $\tilde\sigma_i^{(p)}$ to denote the $i$-th {\em non-zero} eigenvalue of the corresponding operator.

\begin{theorem} 
\label{CompTheorem}
Let $M$ be a compact orientable Riemannian manifold of dimension $n$ with boundary. Then for each $0\leqslant p\leqslant n-2$ ad all $k\in\mathbb{N}$ one has
$$
\tilde\mu_k^{(p)}\leqslant \tilde\sigma_{k}^{(p)}.
$$ 
\end{theorem}
\begin{remark}
Let us note that $c\mathcal C^{n-1}(\partial M) = \mathcal H^{n-1}(\partial M)$ is one-dimensional and from the long exact cohomology sequence of pair $(M,\partial M)$
$$
\ldots\to H^{n-1}(M)\to H^{n-1}(\partial M)\to H^n (M,\partial M)\to H^n(M)\to 0,
$$
one sees that $I_{n-1} =1$, i.~e. $\Lambda\equiv 0$ on $\Omega^{n-1}(\partial M)$.
\end{remark}

Recently there have been several papers~\cite{K, KKK, RS2, RS, SY, SY2, YY, YY2} concerned with estimates for eigenvalues $\tilde\mu_k^{(p)}$. Most proofs of upper bounds in these papers can be modified to yield upper bounds for $\sigma_k^{(p)}$. In a sense, proofs of those bounds implicitly make use of Theorem~\ref{CompTheorem}. In our last theorem, we illustrate that by proving a generalisation of results of Yang and Yu from paper~\cite{YY}.

\begin{theorem}
\label{HPS}
Let $M$ be a compact oriented $n$-dimensional Riemannian manifold with nonempty boundary. Then for any two positive integers $m$ and $r$ and for any $p= 0,\ldots,n-2$, one has
\begin{equation}
\label{HPSineq}
\sigma^{(p)}_{m+I_p}\sigma^{(n-2-p)}_{r+I_{n-2-p}}\leqslant \lambda'^{(p)}_{I_p+m+r+b_{n-p-1}-1},
\end{equation}
where $\lambda'^{(p)}_k$ is the $k$-th eigenvalue of the Hodge-Laplace operator on the space $c\mathcal C^p(\partial M)$.
\end{theorem}

\begin{remark} Theorem of Yang and Yu is obtained from the theorem above by setting $p=0$ and applying Theorem~\ref{CompTheorem} to the left hand side. For details, see Section~\ref{YYHPS}.\end{remark}

\begin{remark}
\label{sharp}
It is shown in Section~\ref{ball} that inequality~(\ref{HPSineq}) is sharp on the Euclidean ball at least for $m,r = 1$. In fact, it is sharp for a wider range of values of $m,r$, see Section~\ref{ball} for details.
\end{remark}

\subsection{Discussion}
\label{discussion}
In this section we discuss a particular case of even $n$ and $p = \frac{n}{2}-1$.

\begin{proposition}
Let $n = 2p+2$ and consider operator $\Lambda$ on the space $\Omega^p(\partial M)$. Then eigenvalues $\sigma^{(p)}_k$ are invariant under conformal changes of metric with conformal factor identically equal $1$ on the boundary. 
\end{proposition}
\begin{proof}
The Rayleigh quotient 
$$
\frac{||d\hat\phi||_{L^2(M)}}{||\phi||_{L^2(\partial M)}}
$$
is invariant under conformal changes of metric described in the statement. 
\end{proof}
The case $n=2$, $p=0$ corresponds to Steklov eigenvalues on surfaces where conformal invariance is well-known.
Moreover, under the same relation between $n$ and $p$ the left hand side of the bound in Theorem~\ref{HPS} only contains the eigenvalues $\sigma^{(p)}$. In particular, setting $m=r$ yields the following theorem.
\begin{theorem}
\label{HPS2}
Let $M$ be a compact oriented $(2p+2)$-dimensional Riemannian manifold with nonempty boundary. Then for any $m>0$ one has the following inequality,
\begin{equation}
\label{HPSineq2}
\left(\sigma^{(p)}_{m+I_{p}}\right)^2\leqslant \lambda'^{(p)}_{I_{p} + b_{p+1} + 2m-1}
\end{equation}
\end{theorem}
The case $n=2$, $p=0$ corresponds to a particular case of Hersch-Payne-Schiffer inequality, which is sharp on the disk for all $m$, see~\cite{GP}.

From explicit computations of $\Lambda$ on the unit ball given in Section~\ref{ball} one can see that inequality~(\ref{HPSineq2}) is sharp on the ball for $m\leqslant \frac{1}{2}{2p+2 \choose p+1}$. It will be interesting to see if unit ball is the unique manifold with this property.
\begin{conjecture}
Suppose that for manifold $M$ inequality~(\ref{HPSineq2}) becomes an equality for $m\leqslant \frac{1}{2}{2p+2 \choose p+1}$. Then $M$ is a Euclidean ball. 
\end{conjecture}

Moreover, it seems that using methods similar to the ones developed in~\cite{GP}, it is possible to show that the inequality in Theorem~\ref{HPS2} is sharp on the ball for all values of $m$. We formulate it as a conjecture.
\begin{conjecture}
Inequality~(\ref{HPSineq2}) is sharp for all values of $m$. To be more precise, for any $m$ and $p$ there exists a sequence $M_k$ of orientable Riemannian manifolds with boundary such that the left hand of inequality~(\ref{HPSineq2}) tends to the right hand side as $k\to\infty$. Moreover, manifolds $M_k$ can be chosen to be a collection of $N = N(m,p)$ euclidean balls of equal radii glued together in the right way.
\end{conjecture}

Previous remarks indicate that eigenvalues $\sigma^{(p)}$ for $(2p+2)$-dimensional manifold $M$ have a lot of features similar to Steklov eigenvalues for surfaces. There is a vast literature devoted to the geometric optimisation problem for Steklov eigenvalues, see e.g.~\cite{GP,GPsurvey,FS1,FS2,K}. Here we propose a similar problem for eigenvalues $\sigma^{(p)}$. Fix an oriented closed Riemannian manifold $(\Sigma,h)$ of dimension $2p+1$. Assume that orientable bordism class of $\Sigma$ is trivial,~i.~e. there exists an orientable manifold $W$ such that $\partial W = \Sigma$. Denote by $[\Sigma,h]_m$ the set of all orientable Riemannian manifolds $(W,g)$ such that $\partial W = \Sigma$, $g|_{\partial W} = h$ and $b_{p+1} = m$. According to Theorem~\ref{HPS2}, for any element of $[\Sigma, h]_m$ the eigenvalue $\sigma^{(p)}_k$ is bounded from above by a quantity depending only on $(\Sigma, h)$ and $m$. For fixed $k,m$ it would be interesting to understand the quantity 
$$
\sup\limits_{[\Sigma,h]_m}\sigma^{(p)}_k.
$$ 
As we pointed out above, for $(\Sigma,h) = (\mathbb{S}^{2p+1},g_{can})$ and $m=0$, Theorem~\ref{HPS2} yields a sharp bound for the first several values of $k$ and the supremum is attained for $(W,g) = (\mathbb{B}^{2p+2},g_{can})$.

\subsection{Organisation of the paper}
The paper is organised in the following way. In Section~\ref{prelim} we show preliminary properties of $\Lambda$ which were essentially demonstrated in~\cite{BS}. In Section~\ref{compactness} we prove that $\Lambda$ is an operator with compact resolvent and Section~\ref{minimax} contains the corresponding variational formulae. Sections~\ref{CT} and~\ref{YYHPS} are devoted to proofs of Theorem~\ref{CompTheorem} and Theorem~\ref{HPS} respectively. Finally, in Section~\ref{ball} we compute the eigenbasis of $\Lambda$ in the case of the unit ball in $\mathbb{R}^{n+1}$.

\section{Preliminaries}
\label{prelim}

\subsection{Hodge-Morrey-Friedrichs decomposition} The cornerstone of our considerations is the Hodge decomposition for manifolds with boundary. 
First, let us record an elementary result that can be proved by computation in local coordinates. 

\begin{proposition} 
\label{inbn}
One has the following equalities, 
$$
\bn\delta = \delta\bn;\qquad \bt d = d\bt;\qquad *\bn = \bt*.
$$
Equivalently,
$$
i_n\delta = \pm\delta i_n;\qquad i^* d = i^* d;\qquad *i_n = \pm i^*
$$
\end{proposition}
\begin{remark}
It is possible to calculate the exact signs in the expressions above which will depend on the degree of the form and dimension of the manifold. However, the signs are not needed in the following and would make the exposition more cumbersome.
\end{remark}

This proposition together with Green's formula~(\ref{Green}) clarifies the following theorem.

\begin{theorem}[Hodge-Morrey-Friedrichs decomposition, see~e.g.~\cite{Schwarz}]
Let $M$ be a compact orientable manifold with non-empty boundary. Then the space of differential $p$-forms on $M$ admits the following decomposition into a direct sum
$$
\Omega^p(M) = d\Omega^{p-1}_D(M)\oplus \delta\Omega^{p+1}_N(M)\oplus \mH^p(M).
$$
Note that boundary conditions are taken {\bf before} applying the operator so that $d\Omega^{p-1}_D(M) = \{\omega\in\Omega^p(M)|\, \omega = d\alpha,\,i^*\alpha = 0\}$.
The space of harmonic fields $\mH^p(M)$ can be further decomposed in two different ways,
$$
\mH^p(M) = \mathcal{EH}^p(M)\oplus\mH^p_N(M)
$$
or
$$
\mH^p(M) = c\mathcal{EH}^p(M)\oplus\mH^p_D(M).
$$
Moreover, $\mH^p_N(M)$ is finite dimensional and constitutes the concrete realisation of absolute de Rham cohomology group $H^p(M,\mathbb{R})$,~i.~e. $\mH^p_N(M)\simeq H^p(M,\mathbb{R})$. Similarly, $\mH^p_D(M)$ is the concrete realisation of relative cohomology group $H^p(M,\partial M,\mathbb{R})$.
\end{theorem}

In fact, one can say more on connection between spaces $\mathcal H^p_D(M)$ and $\mathcal H^p_N(M)$. 

\begin{theorem}[DeTurck, Gluck~\cite{Shonkwiler}]
\label{DTG}
Let $M$ be compact orientable Riemannian manifold with nonempty boundary $\partial M$. Then within the space $\Omega^p(M)$,
\begin{itemize}
\item[(a)] $\mH^p_N(M)$ and $\mH^p_D(M)$ meet only at the origin,
\item[(b)] each of those spaces has decomposition into boundary and interior subspaces,
$$
\mH^p_N(M) = c\mathcal E\mathcal H^p_N(M)\oplus\mathcal E_\partial\mathcal H^p_N(M),
$$
$$
\mH^p_D(M) = \mathcal E\mathcal H^p_D(M)\oplus c\mathcal E_\partial\mathcal H^p_D(M),
$$
where $\mE_\partial$($c\mE_\partial$) denotes the spaces of forms $\omega$ such that $i^*\omega$($i_n\omega$) is a closed (coclosed) form on $\partial M$.
\item[(c)]  $c\mathcal E\mathcal H^p_N(M)\perp\mH^p_D(M)$ and $\mathcal E\mathcal H^p_D(M)\perp\mH^p_N(M)$,
\item[(d)] no larger subspace of $\mH_N^p(M)$ is orthogonal to all of $\mH_D^p(M)$ and no larger subspace of $\mH_D^p(M)$ is orthogonal to all of $\mH_N^p(M)$.
\item[(e)] $\dim\mathcal E_\partial\mathcal H^p_N(M) = \dim c\mathcal E_\partial\mathcal H^p_D(M)$.
\end{itemize}  
\end{theorem}

Hodge-Morrey-Friedrichs decomposition (simply Hodge decomposition in the following) can be used to solve boundary problems for differential forms. It is the subject of Schwarz's book~\cite{Schwarz}. Here we collect several results from that book.

\begin{theorem}[\cite{Schwarz}, Theorem 3.1.1, Lemma 3.1.2]
\label{tangent}
The system
\begin{equation}
\left\{
   \begin{array}{rcl}
	d\omega &=& \chi,\\
	\delta\omega &=& 0,\\
	 i^* \omega &=& \phi\\
   \end{array}
\right.
\end{equation}
has a solution iff $d\chi = 0$, $\bt\chi = \bt d\phi$ and for any $\lambda\in\mH^{p+1}_D(M)$
$$
\ll \chi,\lambda\rr = \int\limits_{\partial M}\phi\wedge*\bn\lambda.
$$
The solution is unique up to an element of $\mH^p_D$.
\end{theorem}
As an immediate corollary we obtain the following.
\begin{corollary} One has the following description
\label{iH}
$$
i^*\mathcal H^p(M) = \{\psi\in\mC^p(\partial M)|\, \psi\perp i_n \mathcal H^{p+1}_D\}.
$$
Moreover, $\mathcal E^{p}(\partial M)\subset i^*\mathcal H^p(M)$.
\end{corollary}
\begin{proof}
The equality is a direct consequence of the theorem above. The inclusion follows from the following calculation. For any $d\alpha\in\mathcal E^p(\partial M)$ and any $\lambda\in \mathcal H_D^{p+1}(M)$ one has
$$
\ll d\alpha, i_n\xi\rr = \int_{\partial M}d\alpha\wedge *\bn \lambda = \int_{\partial M}d(\alpha\wedge*\bn\lambda) \pm \int_{\partial M}\alpha\wedge*\bn\delta\lambda = 0, 
$$
where we used Stokes theorem and identities $\bn\delta = \delta\bn$, $\delta\lambda = 0$.
\end{proof}

By applying the Hodge $*$-operator to the statement of Theorem~\ref{tangent} one obtains the next theorem.

\begin{theorem}[\cite{Schwarz}, Corollary 3.1.3] 
\label{normal}
The system
\begin{equation}
\left\{
   \begin{array}{rcl}
	d\omega &=& 0,\\
	\delta\omega &=& \chi,\\
	 i_n\omega &=& \phi\\
   \end{array}
\right.
\end{equation}
has a solution iff $\delta\chi = 0$, $\bn\chi = \bn \delta\phi$ and for any $\lambda\in\mH^{p-1}_N(M)$
$$
\ll \chi,\lambda\rr = -\int\limits_{\partial M}\bt\lambda\wedge*\phi.
$$
The solution is unique up to an element of $\mH^p_N(M)$. 
\end{theorem}

\begin{corollary}
One has the following equalities,
\label{inH}
\begin{equation}
\label{1}
i_n\mathcal H^p(M) = \{\psi\in c\mC^{p-1}(\partial M)|\, \psi\perp i^*\mathcal H^{p-1}_N(M)\}
\end{equation}
\begin{equation*}
i_n\mathcal H^p(M) = (i^*\mathcal H^{p-1}(M))^{\perp}
\end{equation*}
\end{corollary}
\begin{proof}
The first equality is a direct consequence of the theorem above. 

Let us prove the second. Note that 
$i^*\mathcal H^{p-1}(M) = i^*\mathcal E\mathcal H^{p-1}(M) + i^*\mathcal H^{p-1}_N(M)$, where "+" denotes the sum of the subspaces (not necessarily direct). Moreover, $i^*\mathcal E\mathcal H^{p-1}(M)\subset \mathcal E^{p-1}(\partial M)$ and by Corollary~\ref{iH}, $\mathcal E^{p-1}(\partial M)\subset i^*\mathcal H^{p-1}(M)$, therefore 
$$
i^*\mathcal H^{p-1}(M) = \mathcal E^{p-1}(\partial M) + i^*\mathcal H^{p-1}_N(M).
$$

Taking orthogonal complement of both sides yields
$$
(i^*\mathcal H^{p-1}(M))^\perp = (\mathcal E^{p-1}(\partial M))^\perp \cap (i^*\mathcal H^{p-1}_N(M))^{\perp} = c\mathcal C^{p-1}(\partial M)\cap (i^*\mathcal H^{p-1}_N(M))^{\perp},
$$
which is exactly the right hand side of equality~(\ref{1}). 
\end{proof}

\subsection{Properties of the Dirichlet-to-Neumann map} In this section we study elementary properties of the map $\Lambda$.

\begin{proposition}
\label{indtod}
Any solution of 
$$
\left\{
   \begin{array}{rcl}
	\Delta\omega &=& 0,\\
	i^*\delta\omega &=& 0\\
   \end{array}
\right.
$$
satisfies $\delta\omega = 0$. Similarly, any solution of 
$$
\left\{
   \begin{array}{rcl}
	\Delta\omega &=& 0,\\
	i_n d\omega &=& 0\\
   \end{array}
\right.
$$
satisfies $d\omega = 0$.

\end{proposition}
\begin{proof}
To prove the first statement, note that form $\xi = \delta\omega$ satisfies 
$$
\left\{
   \begin{array}{rcl}
	\Delta\xi &=& 0,\\
	\delta\xi &=& 0,\\
	 i^*\xi &=& 0.\\
   \end{array}
\right.
$$
Therefore, by Green's formula 
$$
||d\xi||^2 = \ll\delta d\xi,\xi\rr + \int_{\partial M}\xi\wedge*\bn d\xi = 0,
$$
i. e. $\xi\in\mathcal H^{p-1}_D(M)$ and by construction $\xi\in c\mathcal E\mathcal H^{p-1}(M)$. Since those spaces are orthogonal, $\delta\omega=\xi=0$.

Application of the first statement to the form $*\omega$ yields the second statement.
\end{proof}

In view of this proposition, the requirement $\delta\omega = 0$ for the harmonic extension is equivalent to $i^*\delta\omega = 0$. Thus, equation~(\ref{DtN}) is a particular case of the following theorem.
\begin{theorem}[\cite{Schwarz}, Lemma 3.4.7]
\label{L'}
The system
\begin{equation*}
\left\{
   \begin{array}{rcl}
	\Delta\omega &=& \eta,\\
	i^*\delta\omega &=& \psi,\\
	 i^*\omega &=& \phi\\
   \end{array}
\right.
\end{equation*}
has a solution iff for any $\lambda\in\mH^{p}_D(M)$
$$
\ll \eta,\lambda\rr = \int\limits_{\partial M}\psi\wedge*\bn\lambda.
$$
The solution is unique up to an element of $\mH^p_D(M)$. 
\end{theorem}

The following propositions are proved in~\cite{BS}. However, since the notations in~\cite{BS} slightly differ from ours, the proofs are provided for the sake of completeness.

\begin{proposition}
\label{L}
The space $\mathcal L(\phi)$ of solutions $\phi$ to equation~(\ref{DtN}) is an affine space with an associated vector space $\mH^p_D$. Therefore there exists unique $\lambda(\phi)\in\mathcal L(\phi)$ such that $\lambda(\phi)\perp\mH^p_D$.
\end{proposition}
\begin{proof}
It suffices to check solvability condition in Theorem~\ref{L'} which is obvious as $\eta = 0$ and $\psi = 0$.
\end{proof}
\begin{proposition}
\label{kernel}
$$
\ker \Lambda = i^*\mathcal{H}^p(M)
$$
\end{proposition}
\begin{proof}
The inclusion $i^*\mathcal{H}^p(M)\subset \ker \Lambda$ is obvious.

For the inverse, suppose $\phi\in\ker\Lambda$ and let $\omega\in\mathcal L(\phi)$. Then $\omega$ satisfies $\Delta\omega = 0$ and $i_nd\omega = 0$. Therefore, by Proposition~\ref{indtod}, $d\omega = 0$,~i.~e. $\omega\in\mathcal H^p(M)$
\end{proof}

\begin{proposition}
\label{simmetricity}
Operator $\Lambda$ is symmetric with respect to $L^2$-inner product on $\Omega^p(M)$.
\end{proposition}
\begin{proof}
Let $\phi,\psi\in\Omega^p(\partial M)$, then Green's formula~(\ref{Green}) implies
$$
0 = \int_M\langle \delta d\lambda(\phi),\lambda(\psi)\rangle = \ll d\lambda(\phi),d\lambda(\psi)\rr - \int_{\partial M}\langle \phi,\Lambda\psi\rangle,
$$
i. e. $\ll d\lambda(\phi),d\lambda(\psi)\rr = \ll \phi,\Lambda\psi\rr$. Switching $\phi$ and $\psi$ in the computation above completes the proof.
\end{proof}

%
\subsection{Image of $\Lambda$}
In this section we identify the image of $\Lambda$. From the previous section, one has the following sequence of inclusions
$$
c\mathcal E^p(\partial M) \subset (\ker\Lambda|_{\Omega^p(\partial M)})^{\perp} = (i^*\mathcal H^p(M))^\perp = i_n\mathcal H^{p+1}(M) \subset  c\mathcal C^p(\partial M).
$$
There are two natural ways to look at the domain of $\Lambda$. One can either set the domain to be $c\mathcal C^p(\partial M)$ which reflects intrinsic geometry of $\partial M$ or set it to be $(i^*\mathcal H^p(M))^\perp = i_n\mathcal H^{p+1}(M)$ which emphasises the role of $M$. A nice feature of the latter is that $\Lambda$ is strictly positive on that domain. However, in most of the article we adapt the former convention and consider $\Lambda$ as an operator on $c\mC^p(\partial M)$

From symmetricity it follows that $\im\Lambda\subset (i^*\mathcal H^p(M))^\perp$. In fact, this inclusion is an equality.

\begin{proposition}
\label{bijection}
Operator
\begin{equation}
\label{lambda-1}
\Lambda\colon i_n\mathcal H^{p+1}(M)\to i_n\mathcal H^{p+1}(M)
\end{equation}
is a bijection. 
\end{proposition}
\begin{proof}
It is sufficient to show surjectivity.
Let $\psi\in i_n\mathcal H^{p+1}(M)$. Then $\exists\,\xi\in\Omega^{p+1}(M)$ satisfying
\begin{equation}
\left\{
   \begin{array}{rcl}
	d\xi &=& 0,\\
	\delta\xi &=& 0,\\
	 i_n \xi &=& \psi.\\
   \end{array}
\right.
\end{equation}
According to Hodge decomposition for harmonic fields one can write $\xi = d\beta + \gamma$, where $\beta\in\Omega^p(M)$ and $\gamma\in\mH^{p+1}_N$. Moreover, $\beta$ can be chosen coclosed. Indeed, consider its Hodge decomposition $\beta = d\tilde\alpha + \delta\tilde\beta + \tilde\gamma$, where $d(d\tilde\alpha + \tilde\gamma) = 0$,~i.~e. $d\delta\tilde\beta = d\beta$. Thus, replacing $\beta$ with $\delta\tilde\beta$ does not change $\xi$. Therefore, $\beta$ solves the system
$$
\left\{
   \begin{array}{rcl}
	\Delta\beta &=& 0,\\
	\delta\beta &=& 0,\\
	 i_n d\beta &=& \psi,\\
   \end{array}
\right.
$$
~i.~e. $\Lambda i^*\beta = \psi$.
\end{proof}

In view of this proposition, in the next section we use $\Lambda^{-1}$ to denote the inverse of $\Lambda$ as an operator in~(\ref{lambda-1}).
Our next goal is to prove compactness of $\Lambda^{-1}$ as an operator on the Hilbert space $L^2(i_n\mathcal H^{p+1}(M))$ which together with simmetricity yields discreteness of the spectrum.
\section{Compactness of $\Lambda^{-1}$}

\label{compactness}
In order to prove the compactness of $\Lambda^{-1}$ we would like to use the following theorem from the book~\cite{Schwarz}.

\begin{theorem}[\cite{Schwarz}, Theorem~3.4.9]
\label{Sch}
For any form $\psi\in (i^*\mathcal H^p(M))^\perp$ there exists a unique solution $\omega$ to  
\begin{equation}
\label{inverse}
\left\{
   \begin{array}{rcl}
	\Delta\omega &=& 0,\\
	i^*\delta\omega &=& 0,\\
	 i_nd\omega &=& \psi,\\
   \end{array}
\right.
\end{equation}
orthogonal to the space $\mH^p(M)$. Moreover, that solution satisfies the following Sobolev bounds
\begin{equation}
\label{Sob1}
||\omega||_{H^{s+2}}\leqslant C||\psi||_{H^{s+1/2}}
\end{equation}
for any $s\in\mathbb{Z}_{\geqslant 0}$.
\end{theorem}

However, for our purposes we need inequality~(\ref{Sob1}) for $s=-1$ which is not guaranteed by the theorem above.

\begin{theorem}
\label{Sch'}
For the solution of equation~(\ref{inverse}) one has the following bound
\begin{equation}
\label{Sob2}
||\omega||_{H^1}\leqslant C||\psi||_{H^{-1/2}}.
\end{equation}
\end{theorem}
This theorem is proved below. For now assume that inequality~(\ref{Sob2}) holds.

\begin{theorem}
\label{compact}
Operator 
$$
\Lambda^{-1}\colon L^2((i^*\mathcal H^p(M))^\perp)\to L^2((i^*\mathcal H^p(M))^\perp)
$$
is compact. Moreover, it is a bounded operator from space $H^{s+1/2}((i^*\mathcal H^p(M))^\perp)$ to space $H^{s-1/2}((i^*\mathcal H^p(M))^\perp)$ for all $s\in\mathbb{Z}_{\geqslant 0}$.
\end{theorem}
\begin{proof}
Note that $\Lambda^{-1}(\psi) = P(i^*\omega)$, where $\omega$ is a solution to~(\ref{inverse}) and $P$ is an $L^2$-orthogonal projection from $L^2\Omega^p(\partial M)$ onto $L^2((i^*\mathcal H^p(M))^\perp)$. Since $H^s(\im\delta)\subset H^s((i^*\mathcal H^p(M))^\perp)\subset H^s(\ker\delta)$ and $H^s(\im\delta)\subset H^s(\ker\delta)$ is a finite codimension closed subspace in a closed space for any $s$ (Hodge decomposition theorem for closed manifolds), then $H^{1/2}((i^*\mathcal H^p(M))^\perp)$ is a split subspace. Thus, using~(\ref{Sob2}) and trace formula one has
$$
||\Lambda^{-1}(\psi)||_{H^{1/2}}\leqslant C||i^*\omega||_{H^{1/2}}\leqslant C'||\omega||_{H^1}\leqslant C''||\psi||_{H^{-1/2}}.
$$ 
Bounds for $H^{s+1/2}$ norms with natural $s$ are proved in similar fashion using inequality~(\ref{Sob1}). Compactness of $\Lambda^{-1}$ follows from inclusion $L^2\subset H^{-1/2}$ and compactness of $H^{1/2}\hookrightarrow L^2$.
\end{proof}

This completes the proof of the first part of Theorem~\ref{MainTheorem}. Note that Sobolev bounds for $\Lambda^{-1}$ imply smoothness of $\Lambda$-eigenforms.

\subsection{Proof of Theorem~\ref{Sch'}}

First, let us provide a weak formulation of equation~(\ref{inverse}): for any $\psi\in H^{-1/2}(\Omega^p(\partial M)):=(H^{1/2}(\Omega^p(\partial M)))^*$ such that $(\psi,\cdot)$ is identically zero on $i^*\mathcal H^p(M)$ find $\omega\in H^1(\mH^p(M)^\perp)$ such that for any $\eta\in H^1(\Omega^p(M))$ one has
\begin{equation}
\label{weak}
\int_M(\langle d\omega,d\eta\rangle + \langle \delta\omega,\delta\eta\rangle) = (\psi,i^*\eta),
\end{equation}
where round brackets denote duality pairing. 

First, note that both sides of equation are invariant under transformation $\eta\mapsto\eta + \xi$, where $\xi\in\mathcal H^p(M)$. Therefore, without loss of generality $\eta\perp_{L^2}\mathcal H^p(M)$. By Lemma 2.4.10.(i) in~\cite{Schwarz} the left hand side of equation~(\ref{weak}) defines a scalar product on $H^1(\mathcal H^p(M)^{\perp})$ equivalent to the usual $H^1$-scalar product. Moreover, right hand side is a bounded linear functional on $H^1(\Omega^p(M))$ as by trace formula
$$
|(\psi,i^*\eta)|\leqslant ||\psi||_{H^{-1/2}}||i^*\eta||_{H^{1/2}}\leqslant C||\psi||_{H^{-1/2}}||\eta||_{H^1}.
$$
Thus, by Riesz representation theorem, there exists solution $\omega$ to~(\ref{weak}) satisfying bound~(\ref{Sob2}).

Easy application of Green's formula shows that if solution $\omega$ is $H^2$ then it is a strong solution in the sense of Theorem~\ref{Sch} and $\psi = i_nd\omega\in H^{1/2}(\Omega^p(\partial M))$.

\section{Min-max principle}
\label{minimax}

The goal of this section is to prove the second half of Theorem~\ref{MainTheorem}, i.~e. to obtain a min-max characterisation of eigenvalues similar to the one for Steklov eigenvalues on functions. By Proposition~\ref{simmetricity}, for $\omega_1\in\mL(\phi_1)$, $\omega_2\in\mL(\phi_2)$ one has
$$
\int\limits_{\partial M} \langle \Lambda\phi_1,\phi_2\rangle = \int\limits_M \langle d\omega_1,d\omega_2 \rangle.
$$
This equality suggests that the Rayleigh quotient for operator $\Lambda$ is a ratio of squares of $L^2$-norms of $d\omega_i$ and $\phi_i$. The following proposition makes it possible to omit the condition $\omega_i\in\mL(\phi_i)$.

\begin{proposition}
\label{minL}
Any form $\omega$ in the space $\mL(\phi)$ minimises the quadratic form $Q(\omega) = ||d\omega||^2_{L^2}$ in the class of $p$-forms $\rho$ on $M$ satisfying $i^*\rho = \phi$.
\end{proposition}
\begin{proof}
First, note that $Q(\omega)$ is constant on $\mL(\phi)$ as $d\mathcal H^p_D(M) = 0$. Thus, it is sufficient to prove that for any $\rho$ with $i^*\rho = \phi$ one has $Q(\rho)\geqslant Q(\omega)$ for some $\omega\in\mL(\phi)$.

Let $\rho$ and $\omega$ be as above. Then $d\rho = d(\rho - \omega) + d\omega$, where $i^*(\rho - \omega) = 0$ and $d\omega\in\mH^p(M)$. Therefore, by Green's formula $d(\rho-\omega)\perp d\omega$ and $Q(\rho) = Q(\rho-\omega) + Q(\omega)\geqslant Q(\omega)$.
\end{proof}

\begin{theorem}[Min-max principle] 
The $k$-th eigenvalue $\sigma_k^{(p)}$ of $\Lambda\colon c\mC^p(\partial M)\to c\mC^p(\partial M)$ can be characterised in the following way
$$
\sigma_k^{(p)} = \max_E \min_{\phi\perp E;\, i^*\hat\phi = \phi} \frac{||d\hat\phi||^2_{L^2(M)}}{||\phi||^2_{L^{2}(\partial M)}},
$$
where $E$ runs over all $(k-1)$-dimensional subspaces of $c\mC^p(\partial M)$. Maximum is achieved for $E = V_{k-1}$, where $V_{k-1}$ is spanned by the first $(k-1)$ eigenforms, $\phi$ being the $k$-th eigenform and $\hat\phi\in\mathcal L(\phi)$. In particular,
$$
\sigma_k^{(p)}\leqslant  \frac{||d\hat\phi||^2_{L^2(M)}}{||\phi||^2_{L^{2}(\partial M)}}
$$
for any $\phi\perp V_{k-1}$ and any $\hat\phi$ satisfying $i^*\hat\phi = \phi$.
%
\end{theorem}
\begin{proof}
Application of min-max theorem for positive self-adjoint operator $\Lambda$ guarantees that 
$$
\sigma_k^{(p)} = \max_{E}\min_{\phi\perp E} \frac{||d\lambda(\phi)||^2_{L^2(M)}}{||\phi||^2_{L^{2}(\partial M)}},
$$
where $E$ runs over all $(k-1)$-dimensional subspaces of $H^{1/2}(c\mC^p(\partial M))$.
Elliptic regularity estimates of Theorem~\ref{compact} guarantee that it is sufficient to consider $E\subset c\mC^p(\partial M)$. Therefore, the min-max formula of the theorem follows from Proposition~\ref{minL}.
\end{proof}


\section{Proof of Theorem~\ref{CompTheorem}}
\label{CT}

Let us remind the reader a definition of operator $L$ defined by Raulot and Savo in~\cite{RS}. By~\cite{Schwarz} Theorem 3.4.10, for any $\phi\in\Omega^p(\partial M)$ there exists unique $\hat\omega\in\Omega^p(M)$ satisfying 
\begin{equation}
\label{RS}
\left\{
   \begin{array}{rcl}
	\Delta\hat\omega &=& 0,\\
	i_n\hat\omega &=& 0,\\
	i^*\hat\omega &=& \phi.\\
   \end{array}
\right.
\end{equation}
Then $L\phi$ is defined to be $i_nd\hat\omega$. In~\cite{RS} the authors demonstrated that $L$ is an elliptic, self-adjoint pseudodifferential operator of first order. Therefore, its spectrum consists of eigenvalues which will be denoted by 
$$
0\leqslant\mu_1^{(p)}\leqslant\mu_2^{(p)}\leqslant\ldots
$$
The kernel of this map is the space $i^*\mathcal H^p_N(M)$. Eigenvalues $\mu_k^{(p)}$ have min-max characterisation which is the subject of the next theorem.
\begin{theorem}[Min-max principle~\cite{RS}]
The $k$-th eigenvalue $\mu_k^{(p)}$ can be computed in the following way
$$
\mu_k^{(p)} = \max_E\min_{\phi\perp E;\, i^*\hat\phi = \phi, \, i_n\hat\phi = 0}\frac{||d\hat\phi||^2_{L^2(M)} + ||\delta\hat\phi||^2_{L^2(M)}}{||\phi||^2_{L^2(\partial M)}},
$$
where $E$ runs over $(k-1)$-dimensional subspaces of $\Omega^p(\partial M)$. Maximum is achieved for $E=V_{k-1}$, where $V_k$ is spanned by the first $(k-1)$-eigenforms, $\phi$ being the $k$-th eigenform and $\hat\phi$ is a solution to~(\ref{RS}). In particular,
$$
\mu_k^{(p)}\leqslant \frac{||d\hat\phi||^2_{L^2(M)} + ||\delta\hat\phi||^2_{L^2(M)}}{||\phi||^2_{L^2(\partial M)}}
$$
for any $\phi\perp V_{k-1}$ and $i^*\hat\phi = \phi, \, i_n\hat\phi = 0$.
\end{theorem}

We turn to Theorem~\ref{CompTheorem}. Let us remind the statement .
\begin{theorem}
Let $\tilde\sigma^{(p)}_k$ and $\tilde\mu^{(p)}_k$ denote the $k$-th non-zero eigenvalue of $\Lambda$ and $L$ respectively. Then for any $0\leqslant p \leqslant (n-2)$
$$
\tilde\mu_k^{(p)}\leqslant\tilde\sigma_k^{(p)}
$$
\end{theorem}

Just for the record, let us state the same inequality for eigenvalues without the tilde.
\begin{corollary} One has the following inequality
$$
\mu_{k+b_p}\leqslant\sigma_{k+I_p},
$$
where $b_p = \dim H^p(M)$ and $I_p = \dim\mathrm{im}\{i_p\colon H^p(M)\to H^p(\partial M)\}$.
\end{corollary}

We start the proof with a couple of preliminary results.

\begin{proposition}
\label{existence}
For any $\phi\in\mE^p(\partial M)$ there exists $\xi\in\mathcal E\mathcal H^p(M)$ satisfying $i^*\xi = \phi$.
\end{proposition}
\begin{proof}
Let $\phi = d\alpha$, then $\xi = d\lambda(\alpha)$ is the form in question. Indeed, $i^*\xi = di^*\lambda(\alpha) = d\alpha = \phi$ and $\delta \xi = \delta d\lambda(\alpha) = \Delta\lambda(\alpha) = 0$.
%
%
%
\end{proof}

\begin{proposition}
\label{RSBS}
For any $\phi\in\Omega^p(\partial M)$ there exists (not necessarily unique) $\psi\in\Omega^p(\partial M)$ such that $\psi-\phi\in i^*\mathcal H^p(M)$, $\psi\perp i^*\mathcal H^p_N(M)$ and there exists a solution $\omega$ to
\begin{equation}
\label{musigma}
\left\{
   \begin{array}{rcl}
	\Delta\omega &=& 0,\\
	\delta\omega &=& 0,\\
	i_n\omega &=& 0,\\
	i^*\omega &=& \psi.\\
   \end{array}
\right.
\end{equation}

\end{proposition}
\begin{proof}
By Proposition~\ref{existence} there exists $\chi\in\mathcal {EH}^{p+1}(M)$ such that $i^*\chi = d\phi$ and $\chi$ is unique up to $\mathcal{EH}^{p+1}_D(M)$. Let $\omega'$ be a primitive of $\chi$, i.e. $d\omega' = \chi$. Consider Hodge decomposition $\omega' = d\alpha + \delta\beta + \gamma$, then $\omega = \delta\beta +\gamma_N$ solves
\begin{equation}
\left\{
   \begin{array}{rcl}
	\Delta\omega &=& 0,\\
	\delta\omega &=& 0,\\
	i_n\omega &=& 0\\
   \end{array}
\right.
\end{equation}
for any $\gamma_N\in\mH^p_N(M)$. Set $\omega_\chi$ to be a unique choice of $\gamma_N$ such that $i^*\omega_\chi\perp i^*\mathcal H^p_N(M)$. Consider the space $W = \{i^*\omega_\chi - \phi\,|\, \chi\in\mathcal {EH}^{p+1}(M),\, i^*\chi = d\phi \}$. Then one has the following properties.
\begin{itemize}
\item[1)] The space $W$ is an affine space of dimension $\dim\mathcal{EH}^{p+1}_D(M)$. Indeed, if $i^*\omega_{\chi_1} = i^*\omega_{\chi_2}$ then $\omega_{\chi_1} - \omega_{\chi_2}$ is a harmonic form with zero tangent and normal parts on the boundary. By Green's formula, $\omega_{\chi_1} - \omega_{\chi_2}\in\mH^p_N(M)\cap\mH^p_D(M)$, therefore, it is zero by Theorem~\ref{DTG}(a). 
\item[2)] Therefore, there exists $\phi_0\in W$ such that $\phi_0\perp i_n \mathcal{EH}^{p+1}_D(M)$.
\item[3)] Since $W\subset \mC^p(\partial M)$, Corollary~\ref{iH} and Theorem~\ref{DTG}(b) imply that $\phi_0\in i^*\mathcal H^p(M)$.
\item[4)] By definition, $\phi + W\perp i^*\mathcal H^p_N(M)$. Thus $\psi = \phi+\phi_0$ satisfies all the requirements of the theorem.
\end{itemize}
\end{proof}
\begin{proof}[Proof of Theorem~\ref{CompTheorem}] 

The idea is that if for $\psi$ there exists a solution to equation~(\ref{musigma}) then $\Lambda(\psi) = L(\psi)$ which allows us to connect operators $\Lambda$ and $L$.

Let $V_k$ be the space spanned by the eigenforms of $\Lambda$ corresponding to the first $k$ non-zero eigenvalues, i.e. $V_k$ is spanned by $\phi_1,\ldots,\phi_k$, where $\Lambda\phi_k = \tilde\sigma_k^{(p)}$. In particular, $V_k\perp i^*\mathcal H^p(M)$. Let $\psi_i$ be forms constructed from $\phi$ by means of applying Proposition~\ref{RSBS} and set $\tilde V_k$ be a vector space spanned by $\psi_1,\ldots,\psi_k$. Then Proposition~\ref{RSBS} implies the following properties of $\tilde V_k$:
\begin{itemize}
\item[(i)] for any $\psi\in\tilde V_k$ there exists a solution to~(\ref{musigma});
\item[(ii)] $\tilde V_k\perp i^*\mathcal H^p_N(M)$;
\item[(iii)] if $\psi = \sum_{i=1}^k a_i\psi_i\in\tilde V_k$ then $\phi=\sum_{i=1}^k a_i\phi_i\in V_k$ 
satisfies $\phi-\psi\in i^*\mathcal H^p(M)$. If there exist non-trivial $a_i$'s such that $\psi = 0$ then $\phi\in i^*\mathcal H^p(M)$. But $V_k\perp i^*\mathcal H^p(M)$, therefore, the map $\sum_{i=1}^k a_i\psi_i\mapsto \sum_{i=1}^k a_i\phi_i$ is an isomorphism;
\item[(iv)] $\dim\tilde V_k = k$.    
\end{itemize}

By property (iv), there exists $\psi\in\tilde V_k$ orthogonal to the first $k-1$ eigenforms of $L$ corresponding to non-zero eigenvalues. By property (ii), $\psi\perp\ker L$ and by property (iii), there exists $\phi\in V_k$ such that  $\psi - \phi\in \ker \Lambda$. Let $\hat\psi\in\mathcal L(\psi)$ be the solution to~(\ref{musigma}) and let $\hat\phi$ belong to $\mL(\phi)$. Then $i^*(d\hat\psi - d\hat\phi) = 0$ and $i_nd(\hat\psi - \hat\phi) = \Lambda(\phi - \psi) = 0$, therefore, $d\hat\psi = d\hat\phi$. The min-max theorem yields the following estimates,
\begin{equation*}
\begin{split}
\tilde\mu_k^{(p)} & \leqslant \frac{||d\hat\psi||^2_{L^2(M)} + ||\delta\hat\psi||^2_{L^2(M)}}{||\psi||^2_{L^2(\partial M)}} = \frac{||d\hat\psi||^2_{L^2(M)}}{||\psi||^2_{L^2(\partial M)}} = \frac{||d\hat\psi||^2_{L^2(M)}}{||\phi||^2_{L^2(\partial M)} + ||\psi - \phi||^2_{L^2(\partial M)}}  \leqslant 
\\ &\leqslant\frac{||d\hat\psi||^2_{L^2(M)}}{||\phi||^2_{L^2(\partial M)}} = \frac{||d\hat\phi||^2_{L^2(M)}}{||\phi||^2_{L^2(\partial M)}}\leqslant \sup\limits_{\phi\in V_k}\frac{||d\lambda(\phi)||^2_{L^2(M)}}{||\phi||^2_{L^2(\partial M)}} = \tilde\sigma_k^{(p)}.
\end{split}
\end{equation*} 
\end{proof}

\section{Proof of Theorem~\ref{HPS}}
\label{YYHPS}

In article~\cite{YY} Yang and Yu used the concept of conjugate harmonic forms to generalise the famous result of Hersch, Payne and Schiffer~\cite{HPS}. They proved the following theorem. 

\begin{theorem}[Yang, Yu~\cite{YY}]
\label{ThYY}
Let $M$ be a compact oriented $n$-dimensional Riemannian manifold with nonempty boundary. Let $\lambda_m$ be the $m$-th eigenvalue for the Laplacian operator on $\partial M$. Then for any two positive integers $m$ and $r$, one has 
\begin{equation*}
\mu_{m+1}^{(0)}\mu^{(n-2)}_{b_{n-2}+r}\leqslant \lambda_{m+r+b_{n-1}}.
\end{equation*}
\end{theorem}

Let $\lambda'^{(p)}_k$ denote the $k$-th eigenvalue of the Hodge Laplacian $\Delta_\partial$ on $\partial M$ restricted to the space $c\mathcal C^p(\partial M)$.
We will prove the following.

\begin{theorem}
Let $M$ be a compact oriented $n$-dimensional Riemannian manifold with nonempty boundary. Then for any two positive integers $m$ and $r$ and for any $p= 0,\ldots,n-2$, one has
\begin{equation}
\sigma^{(p)}_{m+I_p}\sigma^{(n-2-p)}_{r+I_{n-2-p}}\leqslant \lambda'^{(p)}_{I_p+m+r+b_{n-p-1}-1}.
\end{equation}
\end{theorem}
\begin{proof}
Let us recall a general construction of conjugate harmonic forms. Let $\phi\perp i^*\mathcal H^p(\partial M)$, then $\xi = *d\lambda(\phi)\in \mathcal H^{n-p-1}(M)$. Suppose that $\phi$ is such that $\xi\perp\mathcal H^{n-p-1}_N(M)$, then by Hodge decomposition theorem $\xi$ is exact. Let $\rho_0$ be a primitive of $\xi$ and let its Hodge decomposition be $\rho_0 = d\alpha + \delta\beta + \gamma$, where $\gamma\in\mH^{n-2-p}(M)$ and $\beta\in\Omega^{n-p-1}_N(M)$. 
There exists $\gamma_0\perp\mH_D^{n-p-2}(M)$ such that $i^*(\delta\beta +\gamma_0)\perp i^*\mathcal H^{n-p-2}(M)$. We call $\psi = i^*(\delta\beta +\gamma_0)$ the dual form to $\phi$ and $\rho = \delta\beta + \gamma_0$ (which, as one can easily see, coincides with $\lambda(\psi)$) the harmonic conjugate of $\lambda(\phi)$. 
\begin{lemma}
The duality map $\phi\mapsto\psi$ is well-defined, linear and injective. 
\end{lemma}
\begin{proof}
From the construction, $\psi$ is dual to $\phi$ iff $*d\lambda(\phi) = d\lambda(\psi)$. If $\psi_1$ and $\psi_2$ are both dual to $\phi$, then $d(\lambda(\psi_1) - \lambda(\psi_2)) = 0$, i.e. $\psi_1 - \psi_2\in\ker\Lambda$. At the same time, $(\psi_1-\psi_2)\perp\ker\Lambda$, therefore $\psi_1 = \psi_2$. Linearity is obvious. 

Let us prove injectivity. If $0$ form is dual to $\phi$ then $d\phi = 0$ and similar arguments as above assert that $\phi=0$.
\end{proof}

Suppose that $\psi$ is dual to $\phi$, then
\begin{equation}
\label{HPSeq}
||d\lambda(\psi)||^4_{L^2(M)} = \left(\int\limits_{\partial M}\langle\psi,i_n d\lambda(\psi)\rangle \right)^2\leqslant\int\limits_{\partial M}|\psi|^2\int\limits_{\partial M}|i_nd\lambda(\psi)|^2 = \int\limits_{\partial M}|\psi|^2\int\limits_{\partial M}|d\phi|^2,
\end{equation}
where we used Green's formula, Cauchy-Schwarz inequality and equality $i_nd\lambda(\psi) = i_n*d\lambda(\phi) = \pm*i^*d\lambda(\phi) = \pm d\phi$.

Let $\phi_i$ be the eigenforms of $\Delta_{\partial}$. Since the kernel of Hodge Laplacian is the space of harmonic $p$-forms on $\partial M$, one can choose $\phi_i$ to satisfy $\phi_1,\ldots,\phi_{I_p}\in\ker\Lambda$, $\phi_j\perp\ker\Lambda$ for $j>I_p$.
Let $\psi_i^{(q)}$ be eigenforms of $\Lambda$ on $c\mC^q(\partial M)$. Let $\phi$ belong to the space $\mathrm{span}\{\phi_{I_p+1},\ldots,\phi_{I_p+m+r-1+b_{n-p-1}}\}$ such that $\phi\perp\mathrm{span}\{\psi^{(p)}_{I_p+1},\ldots,\psi^{(p)}_{I_p+m-1}\}$ and $*d\omega\perp \mathcal H^{n-p-1}_N(M)$. The latter guarantees the existence of the form $\psi$ dual to $\phi$. Moreover, $\phi$ can be chosen so that $\psi\perp\mathrm{span}\{\psi^{(n-p-2)}_{I_{n-p-2}+1},\ldots,\psi^{(n-p-2)}_{I_{n-p-2}+r-1}\}$. By dimension count, it is easy to see that such $\phi$ exists. Then by min-max principles for $\Lambda$ and $\Delta_\partial$ and inequality~(\ref{HPSeq}) one has
\begin{equation*}
\begin{split}
\sigma^{(p)}_{m+I_p}\sigma^{(n-2-p)}_{r+I_{n-2-p}}\leqslant & \frac{||d\lambda(\phi)||^2_{L^2(M)}||d\lambda(\psi)||^2_{L^2(M)}}{||\psi||^2_{L^2(\partial M)}||\phi||^2_{L^2(\partial M)}} =\\ &\frac{||d\lambda(\psi)||^4_{L^2(M)}}{||\psi||^2_{L^2(\partial M)}||\phi||^2_{L^2(\partial M)}}\leqslant \frac{||d\phi||^2_{L^2(\partial M)}}{||\phi||^2_{L^2(\partial M)}}\leqslant  \lambda'^{(p)}_{I_p+m+r-1+b_{n-p-1}},
\end{split}
\end{equation*}
where in the first equality we used the isometry property of Hodge star and equality $*d(\lambda(\phi)) = d\lambda(\psi)$.
\end{proof}

The combination of Theorem~\ref{CompTheorem} and Theorem~\ref{HPS} yields the following generalisation of Theorem~\ref{ThYY}.
\begin{corollary}
Let $M$ be a compact oriented $n$-dimensional Riemannian manifold with nonempty boundary. Then for any two positive integers $m$ and $r$ and for any $p= 0,\ldots,n-2$, one has
\begin{equation}
\mu^{(p)}_{m+b_p}\mu^{(n-2-p)}_{r+b_{n-2-p}}\leqslant \lambda'^{(p)}_{I_p+m+r+b_{n-p-1}-1}.
\end{equation}
\end{corollary}
Note that $I_0=1$, so for $p=0$ this corollary yieldes the statement of Theorem~\ref{ThYY}.

\section{Eigenvalues of the unit Euclidean ball $\mathbb{B}^{n+1}$}

\label{ball}
In this section we compute eigenbasis and eigenvalues for $\Lambda$ on $\mathbb{S}^n = \partial\mathbb{B}^{n+1}$.
We follow article~\cite{RS2} where Raulot and Savo computed eigenspaces and eigenvalues for operator $L$ on $\mathbb{S}^n = \partial\mathbb{B}^{n+1}$. Note that in order to preserve notations from~\cite{RS2} we deviate from the convention that the ambient manifold has dimension $n$ and instead in this section the ambient manifold has dimension $n+1$.
In case of the ball $\mathbb{B}^{n+1}$ operators $L$, $\Lambda$ and $\Delta$ have common basis of eigenforms which we describe below.

Let $P_{k,p}$ denote the space of homogeneous polynomial $p$-forms of degree $k$ in $\mathbb{R}^{n+1}$. We introduce the following subspaces of $P_{k,p}$,
\begin{itemize}
\item $H_{k,p} = \{\omega\in P_{k,p}|\, \Delta_{\mathbb{R}^{n+1}}\omega = 0, \delta_{\mathbb{R}^{n+1}}\omega = 0\}$;
\item $H'_{k,p} = \{\omega\in H_{k,p}|\, d_{\mathbb{R}^{n+1}}\omega=0\}$;
\item $H''_{k,p} = \{\omega\in H_{k,p}|\, i_n\omega = 0\}$.
\end{itemize}
Assume $1\leqslant p \leqslant (n-1)$. Then $\mathcal H^p(\mathbb{S}^n) = 0$ and $\Omega^p(\mathbb{S}^n) = \mathcal E^p(\mathbb{S}^n)\oplus c\mathcal E^p(\mathbb{S}^n)$. It was shown in~\cite{IT} that $\mathcal E^p(\mathbb{S}^n) = \oplus_k (i^*H'_{k,p})$, $c\mathcal E^p(\mathbb{S}^n) = \oplus_k(i^* H''_{k,p})$ and $\delta\colon i^*H'_{k,p}\to i^*H''_{k+1,p-1}$ is an isomorphism. Thus, $\dim i^*H''_{1,p} = \dim i^* H'_{0,p+1} = {n+1\choose p+1}$ as all forms with constant coefficients lie in $H'_{0,p+1}$.

 We see that $H'_{k,p}\subset \mathcal H^p(\mathbb{B}^{n+1})$, therefore $\Lambda$ is identically zero on each $i^*(H'_{k,p})$. Moreover, for $\phi\in i^*(H''_{k,p})$ the form $\lambda(\phi)$ satisfies $i_n\lambda(\phi) = 0$, therefore, $L(\phi) = \Lambda(\phi)$. 

We summarise observations above and results of~\cite{IT, RS2} in the following theorem.
\begin{theorem}
\label{eigenspace}
Spaces $i^*H'_{k-1,p}$ and $i^*H''_{k,p}$ for $k\geqslant 1$ form common eigenbasis of  $\Lambda$, $L$ and $\Delta$. The corresponding eigenvalues are given below.
\begin{itemize}
\item If $\phi\in i^*H'_{k-1,p}$ then $\Lambda\phi = 0$, $L\phi = (k+p-1)\frac{n+2k+1}{n+2k-1}\phi$ and $\Delta\phi = (k+p-1)(n+k-p)\phi$.
\item If $\phi\in i^*H''_{k,p}$ then $\Lambda\phi = L\phi = (k+p)\phi$ and $\Delta\phi = (k+p)(n+k-p-1)\phi$.
\end{itemize}
 \end{theorem}

This theorem implies sharpness properties of inequality~(\ref{HPSineq}) stated in Section~\ref{discussion} and Remark~\ref{sharp}. Indeed, according to Theorem~\ref{eigenspace} inequality~\ref{HPSineq} is sharp for $m=r=1$. Moreover, it is sharp as long as eigenvalues involved coincide with the first eigenvalue. Statement after Theorem~\ref{HPS2} follows from the fact that the multiplicity of $\sigma^{(p)}_1$ and $\lambda'^{(p)}_1$ is equal to $\dim i^*H''_{1,p} = \dim i^*H'_{0,p+1} = \dim H'_{0,p+1} = {n+1 \choose p+1}$.

\subsection*{Acknowledgements} The author is grateful to D. Jakobson, N. Nigam, I. Polterovich, A. Savo and A. Strohmaier for fruitful discussions. Part of this project was completed when the author was visiting N. Nigam at Simon Fraser University. Its hospitality is greatly acknowledged.

This research was partially supported by Tomlinson Fellowship. This work is a part of the author’s PhD thesis at McGill University under the supervision of Dmitry Jakobson and Iosif Polterovich.

\end{document}